\documentclass[a4paper,11pt]{amsart}
\usepackage{amsfonts}
\usepackage{amsmath}
\usepackage{amssymb}
\usepackage[a4paper]{geometry}
\usepackage{mathrsfs}
\usepackage{xcolor}
\usepackage{amsfonts,amsmath,amssymb,amsthm,stmaryrd,hyperref,bbm}
\usepackage{hyperref}
\renewcommand\eqref[1]{(\ref{#1})} 
\usepackage[utf8]{inputenc}
\usepackage[T1]{fontenc}

\numberwithin{equation}{section}
\theoremstyle{plain}
\newtheorem{theorem}{Theorem}[section]

\newtheorem{lemma}[theorem]{Lemma}

\theoremstyle{definition}
\newtheorem{definition}[theorem]{Definition}
\newtheorem{remark}[theorem]{Remark}

\begin{document}
\title[Parabolic equations with convolution nonlinearities]
{Parabolic problems whose Fujita critical exponent is not given by scaling}

\author[A. Z. Fino]{Ahmad Z. Fino}
\address{
  Ahmad Z. Fino
  \endgraf College of Engineering and Technology, \endgraf American University of the Middle East, Kuwait.}
\email{ahmad.fino@aum.edu.kw}
  
\author[B. T. Torebek ]{Berikbol T. Torebek} \address{Berikbol T. Torebek  \endgraf Institute of
Mathematics and Mathematical Modeling \endgraf 28 Shevchenko str.,
050010 Almaty, Kazakhstan.} \email{torebek@math.kz}

\thanks{Corresponding author: Berikbol T. Torebek, email: torebek@math.kz}



\begin{abstract} 
This paper investigates the (fractional) heat equation with a nonlocal nonlinearity involving a Riesz potential:
\begin{equation*}
u_{t}+(-\Delta)^{\frac{\beta}{2}} u= I_\alpha(|u|^{p}),\qquad x\in \mathbb{R}^n,\,\,\,t>0,
\end{equation*}
where $\alpha\in(0,n)$, $\beta\in(0,2]$, $n\geq1$, $p>1.$
We introduce the Fujita-type critical exponent $p_{\mathrm{Fuj}}(n,\beta,\alpha)=1+(\beta+\alpha)/(n-\alpha)$, which characterizes the global behavior of solutions: global existence for small initial data when $p>p_{\mathrm{Fuj}}(n,\beta,\alpha),$ and finite-time blow-up when $p\leq p_{\mathrm{Fuj}}(n,\beta,\alpha)$.

It is remarkable that the critical Fujita exponent is not determined by the usual scaling argument that yields $p_{sc}=1+(\beta+\alpha)/n$, but instead arises in an unconventional manner, similar to the results of Cazenave et al. [Nonlinear Analysis, 68 (2008), 862-874] for the heat equation with a nonlocal nonlinearity of the form $\int_0^t(t-s)^{-\gamma}|u(s)|^{p-1}u(s)ds,\,0\leq \gamma<1.$

The result on global existence for $p>p_{\mathrm{Fuj}}(n,2,\alpha),$ provides a positive answer to the hypothesis proposed by Mitidieri and Pohozaev in [Proc.
Steklov Inst. Math., 248 (2005) 164–185]. We further establish global nonexistence results for the above heat equation, where the Riesz potential term $I_\alpha(|u|^{p})$ is replaced by a more general convolution operator $(\mathcal{K}\ast |u|^p),\,\mathcal{K}\in L^1_{loc}$, thereby extending the Mitidieri–Pohozaev's results established in the aforementioned work. 

Proofs of the blow-up results are obtained using a nonlinear capacity method specifically adapted to the structure of the problem, while global existence is established via a fixed-point argument combined with the Hardy–Littlewood–Sobolev inequality.
\end{abstract}

\maketitle
\tableofcontents
\section{Introduction}

The present paper is devoted to the study of the Fujita critical exponent for an evolution equation involving a Riesz potential term coupled with a power-law nonlinearity
\begin{equation}\label{44}
\left\{\begin{array}{ll}
u_{t}+(-\Delta)^{\frac{\beta}{2}} u= I_\alpha(|u|^{p}),&\qquad x\in \mathbb{R}^n,\,\,\,t>0,\\\\
 u(x,0)=u_{0}(x),& \qquad x\in \mathbb{R}^n,\\
 \end{array}
 \right.
\end{equation}
where $\alpha\in(0,n)$, $\beta\in(0,2]$, $n\geq1$, $p>1$, and the initial data $u_0\in L^\infty(\mathbb{R}^n)$. Here, for $\beta\in(0,2)$, $(-\Delta)^{\beta/2}$  denotes the fractional Laplacian, and  $I_\alpha$ is the Riesz potential of order $\alpha$, given for $f\in L^1_{\mathrm{loc}}(\mathbb{R}^n)$ by
$$I_\alpha f(x) :=A_\alpha(|x|^{-(n-\alpha)}\ast f)=A_\alpha\int_{\mathbb{R}^n}\frac{f(y)}{|x-y|^{n-\alpha}}\,dy,$$
with normalization constant $A_\alpha= \frac{\Gamma\left( \frac{n - \alpha}{2} \right)}{\Gamma\left( \frac{\alpha}{2} \right)\pi^{\frac n2}2^\alpha}$. Moreover, the identity
 $$-\Delta I_2 = \delta,$$
holds, where \( \delta \) denotes the Dirac delta distribution. This implies that \( I_2 \) serves as the Green's function of the Laplacian \( -\Delta \) on \( \mathbb{R}^n \). More generally, the Riesz potential \( I_\alpha \) can be interpreted as the inverse of the fractional Laplacian operator, in the sense that
$$(-\Delta)^{-{\alpha}/{2}}f(x)=I_\alpha f(x).$$
 For further details on these fundamental properties of Riesz potentials, we refer the reader to \cite[Section~I.1]{Landkof}, \cite[Section~V.1.1]{Stein}.\\
In addition, as \( \alpha \to 0 \), the Riesz potential $I_\alpha$ converges to the Dirac delta $\delta$ in the vague sense \cite[p.~46]{Landkof}. On the other hand, as \( \alpha \to n \), one has
\[
I_\alpha  f \to A_n \log \left( \frac{1}{|x|} \right) * f,
\]
for every \( f \in C^\infty_c(\mathbb{R}^n) \) satisfying \( \displaystyle\int_{\mathbb{R}^n} f = 0 \), where the constant $A_n$ is given by
\[
A_n = \lim_{\alpha \to n} (n - \alpha)A_\alpha = \frac{1}{\Gamma\left( \frac{n}{2} \right)\pi^{n/2} 2^{n-1}} \quad \text{\cite[p.~50]{Landkof}}.
\]
The definition of Riesz potentials can be extended from \( \alpha \in (0, n) \) to arbitrary complex \( \alpha \) with \( \mathrm{Re}(\alpha) > 0 \) and \((\alpha-n)/2 \notin \mathbb{N} \cup \{0\} \). In this more general setting, the convolution \( I_\alpha f \) must be interpreted in the sense of distributions, see for instance \cite{Landkof} or \cite[Chapter~2]{Samko}. Throughout this paper, we restrict to the case \( \alpha \in (0, n) \), where the convolution \( I_\alpha f \) is understood in the classical Lebesgue integral sense.

\subsection{Historical background}
In this subsection, we will provide historical background related to the problem \eqref{44}.
\subsubsection{Classical semilinear heat equations} In 1966 (see \cite{Fuj}) Fujita considered the semilinear parabolic Cauchy problem \begin{equation}\label{heat}
\left\{
\begin{array}{ll}
\,\,\displaystyle{u_t-\Delta u=u^{p},}&\displaystyle {x\in {\mathbb{R}^n},\;t>0,}\\
\\
\displaystyle{u(x,0)=u_0(x),}&\displaystyle{x\in {\mathbb{R}^n},}
\end{array}
\right. \end{equation}
to discuss the conditions for the existence of the global positive solutions and proved that: 
\begin{itemize}
    \item[(i)] for any nonnegative $u_0$ the Cauchy problem \eqref{heat} possesses no global positive solutions if $$1<p<p_{\mathrm{Fuj}}(n)=1+\frac{2}{n};$$ 
    \item[(ii)] there exists a positive global solution of \eqref{heat} if $p>p_{\mathrm{Fuj}}(n)$ and $u_0\geq 0$ is smaller than a Gaussian.\end{itemize}
Moreover, the critical case  $p=p_{\mathrm{Fuj}}(n)$ was considered by Hayakawa \cite{11} for $n=1,2$, Sugitani \cite{22} for $n\geq3$. They have been established that the $p=p_{\mathrm{Fuj}}(n)$  also belongs to case (i).
The number $p_{\mathrm{Fuj}}(n)=1+2/n$ is called the critical exponent in the sense of Fujita. 

Later, in \cite{Wes81} Weissler proved that there exists a global solution of the problem \eqref{heat}, if $p>p_{\mathrm{Fuj}}(n)$ and $\|u_0\|_{L^{p_c}(\mathbb{R}^n)}$ is sufficiently small with $p_c=n(p-1)/2>1.$ Mitidieri and Pohozaev in \cite{PM1} extended the results of Fujita, by replacing $u^p$ with $|u|^p$ and $u_0\geq 0$ with $$\liminf\limits_{R\rightarrow\infty}\int_{|x|\leq R}u_0(x)dx\geq 0,$$ and proved that the problem \eqref{heat} possesses no global sign-changing solutions if $p\leq p_{\mathrm{Fuj}}(n)$. Almost simultaneously, Zhang \cite{Zhang} established the same results for $\int_{\mathbb{R}^n}u_0(x)dx> 0$ and $p\leq p_{\mathrm{Fuj}}(n).$
\subsubsection{Fractional extensions}
If, instead of the classical Laplacian operator $-\Delta$ in \eqref{heat}, one considers its fractional power $(-\Delta)^{\frac{\beta}{2}}$ with $\beta\in(0,2)$, that is, 
\begin{equation}
\label{fheat}
u_t + (-\Delta)^{\frac{\beta}{2}} u = u^p, \qquad x \in \mathbb{R}^n,\, t > 0,
\end{equation}
then the corresponding Fujita critical exponent for problem \eqref{fheat} is given by $$p_{\mathrm{Fuj}}(n,\beta)=1+\frac{\beta}{n}.$$ The result was first established by Nagasawa et al. \cite{Naga}; the critical case was later proved by Sugitani \cite{22}, and further extensions were developed in \cite{Fino5, Guedda, Kirane-Guedda}. 

\subsubsection{Parabolic problems with convolution nonlinearities} In \cite{CDW} Cazenave et al. studied the semilinear heat equation with time-nonlocal convolution nonlinearity
\begin{equation}\label{C2}
\begin{array}{ll}
\displaystyle u_t-\Delta u=\int_0^t(t-s)^{-\gamma}|u(s)|^{p-1}u(s)\,ds,&\qquad {x\in \mathbb{R}^n,\,\,\,t>0.}\end{array}
\end{equation}
In particular, they proved that the critical exponent ensuring global existence of solutions to equation \eqref{C2} is given by $$p_{\mathrm{Fuj}}(n,\gamma)=\max\Big\{\frac{1}{\gamma},p_\gamma\Big\}\in(0,+\infty],\qquad\hbox{with}\quad p_\gamma=1+\frac{2(2-\gamma)}{(n-2+2\gamma)_+},$$ 
that is, it has been proven that: if $p \leq p_{\mathrm{Fuj}}(n,\gamma)$, and $u_0$ is nonnegative, then $u$ blows up in finite time, while there exists a global positive solution for sufficiently small $u_0\geq 0$, if $p>p_{\mathrm{Fuj}}(n,\gamma).$ Further developments and generalizations of the above result were examined in many papers (see, for example, \cite{Fino5, KFTS, Loayza}).

Problem \eqref{44} in the case $\beta = 2$ was first studied by Mitidieri and Pohozaev in \cite{PM2}. They proved that if \begin{equation}\label{PM1}    
\liminf\limits_{R\rightarrow\infty}R^{-\alpha}\int_{B_R}u_0(x)dx\geq 0\,\,\,\, \text{and}\,\,\,\,p \leq 1+\frac{2+\alpha}{n-\alpha},\end{equation} then problem \eqref{44} has no nontrivial weak solutions. Here, $B_R$ stands for the closed ball centered at O with radius $R.$
However, the case $$p > 1+(2+\alpha)/(n-\alpha)$$ was not explored in that paper, and the authors \underline{\bf conjectured} that (see \cite[Remark 7]{PM2}) when $p > 1+(2+\alpha)/(n-\alpha),$ there exists a global solution corresponding to initial data $u_0,$ which satisfies $$u_0(x)\leq c(1+|x|)^{-\gamma},\,c>0,\,x\in\mathbb{R}^n,$$ for $\gamma>(2+\alpha)/(p-1)$.

Several years ago, in \cite{Filippucci1, Filippucci2} Filippucci and Ghergu investigated the heat equation with the nonlocal nonlinearity $I_{\alpha}(|u|^{p})|u|^{q}$ and proved that the considered problem admits no global weak solutions whenever the initial data satisfy $$u_0\in L^1(\mathbb{R}^n), \int_{\mathbb{R}^n}u_0(x)\,dx>0$$ and $$2<p+q\leq1+\frac{\alpha+2}{2n-\alpha},\, p,q>0.$$ This result was recently extended in \cite{FT25} by the authors of this paper to the fractional Laplacian
$(-\Delta)^{\frac{\beta}{2}},\,\beta\in (0,2],$ and the nonexistence of global solutions was established for $2<p+q<1+(\alpha+\beta)/n,\,p,q>0$, thereby improving the previous result obtained in the case $\beta=2.$

\subsection{Motivation} Motivated by the above results, this paper poses the following natural questions:
\begin{description}
    \item[Question 1]  Is the Mitidieri–Pohozaev conjecture true? This {\bf conjecture} asserts that for $p>1+(2+\alpha)/(n-\alpha)$, there exists a global solution for sufficiently small initial data. If this statement holds, then $1+(2+\alpha)/(n-\alpha)$ represents the critical Fujita exponent for problem \eqref{44} for $\beta=2$, that is, 
    $$p_{\mathrm{Fuj}}(n,2, \alpha)=1+\frac{2+\alpha}{n-\alpha}.$$ Similar questions arise for all $\beta\in(0,2).$
    \item[Question 2] Is it possible to interpret the nonexistence of nontrivial solutions as a finite-time blow-up phenomenon in the case $p\leq 1+(\alpha+\beta)/(n-\alpha)$? In other words, can one assert the existence of a local continuous solution on $(0,T_{\max})$ that blows up as $t\rightarrow T_{\max}$, that is, $$\lim\limits_{t\rightarrow T_{\max}}\|u(\cdot,t)\|_{X}\rightarrow\infty?$$ Here $X$ is a Banach space. It should be noted that in \cite{PM2}, the nonexistence of  nontrivial solutions was demonstrated, which, in the general case, does not necessarily imply blow-up of the solution in finite time.
    \item[Question 3] It is natural to ask whether the results of Mitidieri and Pohozaev can be {\bf extended} to more general kernel $K(x,y)$. In fact, in \cite{PM2}, Mitidieri and Pohozaev considered problem \eqref{44} with a more general kernel satisfying $K^{-1}\in L^1_{loc}$. However, the nonexistence results were established only in the case of the Riesz potential $I_\alpha$. 
\end{description}
The main purpose of the paper is to answer the above questions.

\subsection{Main results}
This subsection presents the main results of the paper: local existence, blow-up behavior, and global existence of solutions to problem \eqref{44}.

Let us start with the definition of the mild solution of \eqref{44}.

\begin{definition}  A function $u\in L^\infty((0,T),X)$ is called a {\bf mild solution} of $\eqref{44}$ if $u$ has the initial data $u_0$ and satisfies the integral equation
	\begin{equation}\label{branda}
		u(t)=S_{\beta}(t)u_0 + \int_{0}^{t}S_{\beta}(t-\tau ) I_\alpha(|u|^{p})(\tau ) \,\mathrm{d}\tau,
	\end{equation}
where $S_{\beta}(t)$ is defined by \eqref{3.3} below.
\end{definition}
Below, we formulate a statement about local existence.
\begin{theorem}[Local existence]\label{localexistence}
	Let $\beta\in(0,2]$, $\alpha\in(0,n)$, $p>n/(n-\alpha)$, and $u_0 \in L^s(\mathbb{R}^{n})\cap L^\infty(\mathbb{R}^{n})$ with $n/(n-\alpha)<s<n(p-1)/\alpha$. Then there exists a time $T=T(u_0)>0$ such that problem \eqref{44} possesses a unique mild solution 
	$$u\in C([0,T],L^s(\mathbb{R}^n))\cap L^\infty((0,T),L^\infty(\mathbb{R}^{n})).$$
	 Moreover, the following properties hold:
	\begin{itemize}
		\item[$\mathrm{(i)}$] Problem \eqref{44} possesses a maximal mild solution $$u\in C([0,T_{\max}),L^s(\mathbb{R}^{n}))\cap L^\infty((0,T_{\max}),L^\infty(\mathbb{R}^{n}))$$ with $u(0)=u_0$, where $T_{\max}=T_{\max}(u_0)\leq\infty$. Furthermore, either $T_{\max}=\infty$, or else $T_{\max}<\infty$ and
		\begin{equation}\label{limblowup}
			\lim_{t\rightarrow T_{\max}^-}(\|u(t)\|_{L^\infty(\mathbb{R}^{n})}+\|u(t)\|_{L^s(\mathbb{R}^{n})})=+\infty.
		\end{equation}
		\item[$\mathrm{(ii)}$] If $u_0\geq0$, then the mild solution $u\geq0$ remains nonnegative for all $t\in[0,T]$.
	\end{itemize}
\end{theorem}
\begin{remark}
As shown in Theorem \ref{localexistence}, the local existence of a solution to problem (1.1) requires that $p>n/(n-\alpha)$. \textcolor{red}{This restriction arises from the nonlocal nature of the nonlinear source term and, in particular, from the use of Lemma \ref{Hardy} combined with H\"older's inequality, and seems to be of a technical nature.} Consequently, in the case $1<p\leq n/(n-\alpha),$ whether the local solution exists or not remains an open question.
\end{remark}
It is straightforward to verify that if \( u(t, x) \) is a solution of equation \eqref{44} with initial data \( u_0 \), then for all \( \lambda > 0 \), the rescaled function
\[
u_\lambda(t,x)=\lambda^{\frac{\beta+\alpha}{p - 1}} u(\lambda^\beta t, \lambda x)
\]
is also a solution, with initial value \( u_\lambda(0,x)=\lambda^{\frac{\beta+\alpha}{p - 1}} u_0(\lambda x) \). Moreover, the $L^q$-norm of the rescaled initial data satisfies
\[
\left\|u_\lambda (0) \right\|_{L^q} = \lambda^{\frac{\beta+\alpha}{p - 1} - \frac{n}{q}} \|u_0\|_{L^q}.
\]
This observation shows that the scaling-invariant Lebesgue exponent for equation \eqref{44}  is
 \begin{equation}\label{qscaling}
q_{\mathrm{sc}} = \frac{n(p - 1)}{\beta+\alpha}. 
\end{equation}
One might therefore expect that if \( q_{\mathrm{sc}} > 1 \), that is, \( p > p_{\mathrm{sc}} \), where the critical exponent is given by
 \begin{equation}\label{pscaling}
p_{\mathrm{sc}} = 1 + \frac{\beta+\alpha}{n},
\end{equation}
and if \( \|u_0\|_{L^{q_{\mathrm{sc}}}} \) is sufficiently small, then the corresponding solution should exist globally in time. However, the next result shows that this expectation does not hold.
\begin{theorem}\label{global} Let $n\geq1$, $\alpha\in(0,n)$, $0<\beta\leq 2$, and $p>1$. 
\begin{itemize}
\item[(i)] If $u_0\in L^1(\mathbb{R}^n)\cap L^\infty(\mathbb{R}^n)$, $u_0\geq 0$ such that $\displaystyle \int_{\mathbb{R}^n}u_0(x)\,dx>0$, and 
 \begin{equation}\label{esti1}
 \frac{n}{n-\alpha}<p\leq p_{\mathrm{Fuj}}(n,\beta,\alpha):=1+\frac{\beta+\alpha}{n-\alpha},
\end{equation}
then the mild solution of problem \eqref{44}  blows-up in finite
 time.
\item[(ii)] If
 \begin{equation}\label{esti2}
    p> p_{\mathrm{Fuj}}(n,\beta,\alpha),
\end{equation}
and $u_0\in L^{q_{\mathrm{sc}}}(\mathbb{R}^n)\cap L^{\infty}(\mathbb{R}^n)$ (where $q_{\mathrm{sc}}$ is given by \eqref{qscaling}), then the mild solution $u\in L^\infty((0,\infty),L^\infty(\mathbb{R}^n))$ exists
globally in time, provided that $\|u_0\|_{L^{q_{\mathrm{sc}}}}$ is sufficiently small.\end{itemize}
\end{theorem}

\begin{remark} ${}$
\begin{itemize}
\item [$(\mathrm{a})$] Based on \eqref{esti1} and \eqref{esti2}, we identify $p_{\mathrm{Fuj}}(n,\beta,\alpha)$ as the critical exponent for problem \eqref{44}. However, since $p_{\mathrm{Fuj}}(n,\beta,\alpha)>p_{\mathrm{sc}}$, where $p_{\mathrm{sc}}$ is the scaling exponent defined in \eqref{pscaling}, it follows that the critical exponent $p_{\mathrm{Fuj}}(n,\beta,\alpha)$ is {\bf not governed by scaling arguments}. Cazenave et al. \cite{CDW} reported a similar phenomenon while examining the Fujita-type critical exponent for the semilinear heat equation involving  nonlocal nonlinearity of the form \eqref{C2}. 
\item[$(\mathrm{b})$] Part (ii) of Theorem \ref{global} provides a {\bf positive answer to the conjecture} of Mitidieri and Pohozaev posed in \cite{PM2}, which states that if $p>1+(2+\alpha)/(n-\alpha)$ and $u_0$ is sufficiently small, then problem \eqref{44} (when $\beta=2$) admits a global solution. In the case of the fractional Laplacian with $\beta\in(0,2),$ Theorem \ref{global} yields significant new results, identifying a new critical Fujita exponent given by $$p_{\mathrm{Fuj}}(n,\beta,\alpha)=1+\frac{\beta+\alpha}{n-\alpha}.$$ From this expression, it is evident that as $\alpha \to 0$ and $\beta = 2$, the quantity $p_{\mathrm{Fuj}}(n,2,0)$ reduces to the classical Fujita critical exponent $p_{\mathrm{Fuj}}(n)=1+2/n$. Moreover, as $\alpha \to n$, the critical exponent tends to $\infty$.
\item[$(\mathrm{c})$] Observe that one may assume the decay condition $|u_0(x)|\leq C (1+|x|)^{-\gamma}$ with $\gamma>(\beta+\alpha)/(p-1)$ instead of the integrability condition $u_0\in L^{q_{\mathrm{sc}}}(\mathbb{R}^n)$. This assumption is consistent with that made by Mitidieri and Pohozaev \cite{PM2}, who hypothesized that for such $u_0$, the solution exists globally. Moreover, it highlights its relevance to Theorem  \ref{blowup} (ii) below.
\item[$(\mathrm{d})$] When $\beta=2$, the condition $u_0\geq0$ is not required; see, for example, \cite{Fino5}. For $\beta\in(0,2)$, the test function recently used in \cite{DaoFino} is also recommended; however, due to \eqref{triangle inequality}, we refrain from using it.
\textcolor{red}{\item[$(\mathrm{e})$] To better understand the originality of the critical exponent $p_{\mathrm{Fuj}}(n,\beta,\alpha)$, let us note that the steady-state solution $v$ of problem \eqref{44} can formally satisfy
$$(-\Delta)^{\frac{\alpha+\beta}{2}}v=|v|^p$$
where we have used the fact that the Riesz potential $I_\alpha$ can be regarded as the inverse of the fractional Laplacian $(-\Delta)^{\alpha/2}$. By setting $\varphi(x)=|x|^{-(n-\alpha)}$ for $x\in\mathbb{R}^n\setminus\{0\}$ and applying the estimate of Bonforte and V\'azquez \cite{BonforteVazquez}:
$$(-\Delta)^{\frac{\alpha+\beta}{2}}\varphi\leq C|x|^{-(n-\beta)},$$
one can see that $\varphi$ is a steady-state subsolution of problem \eqref{44} if 
$$(n-\alpha)p = n + \beta\qquad \hbox{that is}\qquad p=p_{\mathrm{Fuj}}(n,\beta,\alpha).$$}
\end{itemize}
\end{remark} 

In the following, we consider a more general version of problem \eqref{44}, in which the Riesz potential is replaced by a general convolution with kernel $\mathcal{K}.$

Let
\begin{equation}\label{1}
\left\{\begin{array}{ll}
u_{t}+(-\Delta)^{\frac{\beta}{2}} u= (\mathcal{K}\ast |u|^{p}),&\qquad x\in \mathbb{R}^n,\,\,\,t>0,\\\\
 u(x,0)=u_{0}(x),& \qquad x\in \mathbb{R}^n,\\
 \end{array}
 \right.
\end{equation}
where $\beta\in(0,2]$, $n\geq1$, $p>1$, $u_0\in L^1_{\hbox{\tiny{loc}}}(\mathbb{R}^n)$. The function $\mathcal{K}:(0,\infty)\rightarrow(0,\infty)$ is continuous, it satisfies $\mathcal{K}(|\cdotp|)\in L^1_{\hbox{\tiny{loc}}}(\mathbb{R}^n)$, and there exists $R_0>1$ such that $$\inf\limits_{r\in(0,R)}\mathcal{K}(r)=\mathcal{K}(R),\qquad \text{for all} \,\,\,\,R>R_0.$$
\textcolor{red}{Typical examples of $\mathcal{K}$ are the constant functions as well as
$$\mathcal{K}(r)=r^{-(n-\alpha)},\,\,\alpha\in(0,n)\quad\hbox{or}\quad \mathcal{K}(r)=r^{-(n-\alpha)}\log^\beta(1+r),\,\,\alpha\in(0,n),\,\beta\in\mathbb{R},\,\beta+\alpha>0.$$ 
}The nonlinear convolution term $\mathcal{K}\ast|u|^p$ is the Fourier convolution between $\mathcal{K}$ and $|u|^p$ defined by
\begin{align*}(\mathcal{K}\ast |u|^p)(x)&=\int_{\mathbb{R}^n}\mathcal{K}(|x-y|)|u(y)|^p\,dy\\&=\int_{\mathbb{R}^n}\mathcal{K}(|y|)|u(x-y)|^p\,dy.\end{align*}
First, we give the definition of a weak solution to problem \eqref{1}.
\begin{definition}\textup{(Weak solution of \eqref{1})}${}$\\
Let $u_0\in L^1_{\hbox{\tiny{loc}}}(\mathbb{R}^n)$ and $T>0$. We say that $u\in L_{\hbox{\tiny{loc}}}^1((0,T)\times\mathbb{R}^n)$ is a weak solution of \eqref{1} on $[0,T)\times\mathbb{R}^n$ if
$$ (\mathcal{K}\ast |u|^p) \in L_{\hbox{\tiny{loc}}}^1((0,T)\times\mathbb{R}^n),$$
and
\begin{equation}\label{weaksolution2}\begin{split}
\int_0^\tau\int_{\mathbb{R}^n}(\mathcal{K}\ast|u|^p)\psi(t,x)\,dx\,dt+\int_{\mathbb{R}^n}u(0,x)\psi(0,x)\,dx&=\int_0^\tau\int_{\mathbb{R}^n}u\,(-\Delta)^{\frac{\beta}{2}}\psi(t,x)\,dx\,dt\nonumber\\&-\int_0^\tau\int_{\mathbb{R}^n}u\,\psi_t(t,x)\,dx\,dt,
\end{split}\end{equation}
holds for all compactly supported test function $\psi\in C^{1,2}_{t,x}([0,T)\times\mathbb{R}^n)$, and $0\leq\tau<T$.\\ If $T=\infty$,  $u$ is called a global in time weak solution to \eqref{1}.
\end{definition}
\begin{theorem}\label{theo11}
Let $n\geq1$, $\beta\in(0,2]$, and  $p>1$.
\begin{itemize}
\item[(i)] If $u_0\in L^1(\mathbb{R}^n)$,
$$\int_{\mathbb{R}^n}u_0(x)\,dx>0\qquad\hbox{and}\qquad \limsup_{R\rightarrow\infty}\left(\mathcal{K}(R)\,R^{\frac{n+\beta}{p}}\right)>0,$$
then problem \eqref{1} has no global weak nonnegative solutions.
\item[(ii)] If $u_0\in L^1_{loc}(\mathbb{R}^n)$,
$$\liminf_{R\rightarrow\infty}R^{-n}(\mathcal{K}(R))^{-1}\int_{B_R}u_0(x)\,dx\geq0\qquad\hbox{and}\qquad \limsup_{R\rightarrow\infty}\left(\mathcal{K}(R)\,R^{\frac{n+\beta}{p}}\right)>0,$$
then problem \eqref{1} has no global weak nontrivial nonnegative solutions.
\item[(iii)] If $ u_0\in L^1_{\hbox{\tiny{loc}}}(\mathbb{R}^n)$,
$$u_0(x)\geq\varepsilon(1+|x|^2)^{-\gamma/2}\,\,\,\,\hbox{and}\,\,\,\,\,
\liminf_{R\rightarrow\infty}\left(\mathcal{K}(R)^{-1}R^{\gamma(p-1)-n-\beta}\right)=0,$$
for some positive constant $\varepsilon>0$ and any exponent $\gamma>0$, then problem \eqref{1} has no global weak nonnegative solutions.\\
\end{itemize}
\end{theorem}
\begin{remark}
 For $\beta=2$, nonnegativity of the solutions in Theorem \ref{theo11} is not necessary; see, e.g., \cite{FinoKiraneBK}.
    \end{remark}
Since problem \eqref{1} reduces to problem \eqref{44} when $\mathcal{K}(r)=A_\alpha r^{-(n-\alpha)}$, Theorem \ref{theo11} immediately yields the following result, using the fact that
$$ \liminf_{R\rightarrow\infty}\left(\mathcal{K}(R)^{-1}R^{\gamma(p-1)-n-\beta}\right)=0\,\,\Longleftrightarrow \,\,p<p_*:=1+\frac{\beta+\alpha}{\gamma}.$$

\begin{theorem}\label{blowup}
Let $n\geq1$, $\alpha\in(0,n)$, $0<\beta\leq 2$, and $p>1$.
\begin{itemize}
\item[(i)] If $u_0\in L^1_{loc}(\mathbb{R}^n)$ such that $$\liminf\limits_{R\rightarrow\infty}R^{-\alpha}\int_{B_{R}}u_0(x)\,dx\geq 0,\,\,\,\, \text{and}\,\,\,\,
 1< p\leq p_{\mathrm{Fuj}}(n,\beta,\alpha),$$
then the \eqref{44} has no global nontrivial nonnegative weak solutions.
\item[(ii)] If $u_0\in L^1_{\hbox{\tiny{loc}}}(\mathbb{R}^n)$,
$$u_0(x)\geq\overline{\varepsilon}(1+|x|^2)^{-\gamma/2}\,\,\,\,\hbox{and}\,\,\,\,\,
p<p_*=1+\frac{\beta+\alpha}{\gamma},$$
for some positive constant $\overline{\varepsilon}>0$ and any exponent $\gamma>0$, then problem \eqref{44} has no global weak solutions.\\
\end{itemize}
\end{theorem}
\begin{remark} Theorems \ref{theo11} and \ref{blowup} give us reason to present the following arguments:
\begin{itemize}
\item[a)] Theorem \ref{theo11} complements the work \cite{PM2} for more general convolution operators instead of the Riesz potential where in \cite{PM2} was studied.
\item[b)] It is worth noting that choosing $\gamma<n-\alpha$ allows Theorem \ref{blowup}-(ii) to extend the blow-up result  from Theorem \ref{blowup}-(i). 
\item[c)] Under the weaker assumption $u_0\in L^1_{\hbox{\tiny{loc}}}(\mathbb{R}^n)$ in Theorem \ref{blowup}-(i)—as opposed to the stronger condition $u_0\in L^1(\mathbb{R}^n)\cap L^\infty(\mathbb{R}^n)$ in Theorem \ref{global}-(i)—the conclusion is correspondingly weakened: the solution fails to exist globally in the weak sense, rather than blowing up in finite time. However, the restriction on $p$ of the form $p>n/(n-\alpha)$ is no longer required.
\item[d)] \textcolor{red}{For the sake of clarity, we emphasize the distinction between the finite-time blow-up of mild solutions and the nonexistence of global weak solutions. Finite-time blow-up of mild solutions means that a solution exists on a maximal time interval $[0,T_{\max})$ with $0<T_{\max}<\infty$, but the norm of the solution, in an appropriate functional space, becomes unbounded as $t\rightarrow T_{\max}$. In contrast, the nonexistence of global weak solutions refers to the fact that no solution, even in a weak sense, can be defined for all $t\geq 0$ under the given assumptions.}
\end{itemize}
\end{remark}
In summary, Theorems \ref{localexistence}, \ref{global}, \ref{theo11}, and \ref{blowup} offer comprehensive answers to {\bf Questions 1–3}. The subsequent sections are devoted to detailed proofs of these results. Proofs of blow-up and nonexistence of solutions are based on the test function method, originally proposed by various authors (see, for example, \cite{Baras-Pierre, Baras, PM1, Zhang1}). In our case, we construct test functions within the corresponding classes that are specifically adapted to the nonlocal nature of the problems under consideration. Global existence is obtained through the fixed-point method (see \cite{CDW,QS}), using the Hardy–Littlewood–Sobolev inequality, the fractional heat kernel properties, and various functional and algebraic inequalities.

\section{Preliminaries}
In this section, we present some preliminary knowledge needed in our proofs hereafter. First, we recall that the fundamental solution
$S_\beta=S_\beta(x,t)$ of the linear equation
\begin{equation}\label{linearequ+}
u_t+(-\Delta)^{\beta/2}u=0,\quad \beta\in(0,2],\;
x\in\mathbb{R}^n,\;t>0,
\end{equation}
can be written  via the Fourier transform
as follows
\begin{equation}\label{3.3}
S_\beta(x,t)=\frac{1}{(2\pi)^{n/2}}\int_{\mathbb{R}^n}e^{ix.\xi-t|\xi|^\beta}\,d\xi.
\end{equation}
It is well-known that for each $\beta\in(0,2],$ this function
satisfies
\begin{equation}\label{P_1+}
    S_\beta(1)\in L^\infty(\mathbb{R}^n)\cap
L^1(\mathbb{R}^n),\quad
S_\beta(x,t)\geq0,\quad\int_{\mathbb{R}^n}S_\beta(x,t)\,dx=1,
\end{equation}
\noindent for all $x\in\mathbb{R}^n$ and $t>0.$ Hence, using Young's inequality for the convolution and the following self-similar form  $$S_\beta(x,t)=t^{-n/\beta}S_\beta(xt^{-1/\beta},1),$$ we have
\begin{lemma}[$L^p-L^q$ estimate]\label{Lp-Lqestimate} 
	Let $\beta\in(0,2]$, and $1\leq r\leq q\leq\infty$. Then there exists a positive constant $C$ such that for every $v\in L^r(\mathbb{R}^{n})$, the following inequalities hold
	\begin{eqnarray}
		\|S_\beta(t)\ast v\|_q&\leq &Ct^{-\frac{n}{\beta}(\frac{1}{r}-\frac{1}{q})}\|v\|_r,\qquad t>0,\label{semigroup1}\\
		\|S_\beta(t)\ast v\|_r&\leq& \|v\|_r,\qquad t>0.\label{semigroup2}
	\end{eqnarray}
\end{lemma}
The following Hardy-Littlewood-Sobolev is a particular case of \cite[Theorem~4.3]{LiebLoss}.
\begin{lemma}[Hardy-Littlewood-Sobolev]\label{Hardy}${}$\\
	Let $1< p< r<\infty$ and $0<\alpha<n$ with $1/p+\alpha/n=1+1/r$. Then there exists a positive constant $C=C(n,\alpha,p)>0$ such that for every $f\in L^p(\mathbb{R}^{n})$, the following inequality holds
	\begin{equation}\label{Hardy-Littlewood}
		\||x|^{-\alpha}\ast f\|_r\leq C \|f\|_p.
	\end{equation}
\end{lemma}
\begin{definition}[\cite{Kwanicki,Silvestre}]\label{def4}
\fontshape{n}
\selectfont
Let $s \in (0,1)$. Let $X$ be a suitable set of functions defined on $\mathbb{R}^n$. Then, the fractional Laplacian $(-\Delta)^s$ in $\mathbb{R}^n$ is a non-local operator given by
$$ (-\Delta)^s: \,\,v \in X  \to (-\Delta)^s v(x):= C_{n,s}\,\, p.v.\int_{\mathbb{R}^n}\frac{v(x)- v(y)}{|x-y|^{n+2s}}dy, $$
as long as the right-hand exists, where $p.v.$ stands for Cauchy's principal value and $$C_{n,s}:= \frac{4^s \Gamma(\frac{n}{2}+s)}{\pi^{\frac{n}{2}}\Gamma(-s)}$$ is a normalization constant.
\end{definition}

\begin{lemma}\label{lemma4}\cite[Lemma~2.4]{DaoReissig}
Let $s \in (0,1]$, and $\varphi$ be a smooth function satisfying $\partial_x^2\varphi\in L^\infty(\mathbb{R}^n)$. For any $R>0$, let $\varphi_R$ be a function defined by
$$\varphi_R(x):= \varphi(x/R) \quad \text{ for all } x \in \mathbb{R}^n.$$
Then, $(-\Delta)^s \varphi_R$ satisfies the following scaling properties:
$$(-\Delta)^s \varphi_R(x)= R^{-2s}((-\Delta)^s\varphi)(x/R) \quad \text{ for all } x \in \mathbb{R}^n. $$
\end{lemma}



\section{Local existence. }

\begin{proof}[Proof of Theorem \ref{localexistence}] The proof is divided into several steps.\\
	{\it Step 1. Fixed-point argument.} Let $T>0$ be fixed. We define the Banach space
$$E_T=L^\infty((0,T),L^s(\mathbb{R}^n)\cap L^\infty(\mathbb{R}^{n})).$$
The norm on $E_T$ is defined by
$$\|u\|_{E_T}=\sup_{t \in (0,T)}\|u(t)\|_{L^{s}\cap L^\infty}= \sup_{t \in (0,T)}\left(\|u(t)\|_{L^{s}}+\|u(t)\|_{L^{\infty}}\right).$$
We choose $R\geq  \|u_0\|_{L^s\cap L^\infty}$. In order to use the Banach fixed-point theorem, we introduce the following nonempty complete metric space
	\begin{equation}\label{norm}
		B_T(R)=\{u\in E_T:\,\,\|u\|_{E_T}\leq 2 R\},
	\end{equation}
	equipped with the distance $d(u,v)=\|u-v\|_E $.	For $u\in B_{T}(R)$, we define $\Lambda(u)$ by
	\begin{equation}\label{solop}
		\Lambda(u)(t):=S_{\beta}(t)u_0 + \int_{0}^{t}S_{\beta}(t-\tau )  I_\alpha(|u|^{p})(\tau ) \,\mathrm{d}\tau .
	\end{equation}
	Let us prove that  $\Lambda: B_{T}(R) \rightarrow B_{T}(R)$. Using \eqref{semigroup2}, we obtain for any $u \in B_{T}(R)$, 
	\begin{align*}
		\|\Lambda(u)(t)\|_{L^{s}} & \leq \|S_{\beta}(t)u_0\|_{L^{s}} + \left\| \int_{0}^{t}S_{\beta}(t-\tau)  I_\alpha(|u|^p)(\tau) \,\mathrm{d}\tau \right\|_{L^{s}}  \\
				&\leq \|u_0\|_{L^s}+  \int_{0}^{t} \|I_\alpha (|u|^p)(\tau)\|_{L^s} \, \mathrm{d}\tau,
				\end{align*}	
		for all $t\in (0,T)$. Applying Lemma \ref{Hardy} with the identity $1/q=\alpha/n+1/s$, and noting that the condition $n/(n-\alpha)<s$ guarantees $q>1$, we get
		\begin{align*}
		\|\Lambda(u)(t)\|_{L^{s}} &\leq \|u_0\|_{L^s}+   C\int_{0}^{t} \||u(\tau)|^p\|_{L^{q}} \, \mathrm{d}\tau.
				\end{align*}	
			Since $s<n(p-1)/\alpha$ implies $1/q<p/s$, then $p>s/q$ and therefore
		\begin{align}\label{fix1}
		\|\Lambda(u)(t)\|_{L^{s}} &\leq \|u_0\|_{L^s}+  C \int_{0}^{t} \|u(\tau)\|^{p-\frac{s}{q}}_{L^{\infty}} \||u(\tau)|^{\frac{s}{q}}\|_{L^{q}} \, \mathrm{d}\tau\notag\\
		&= \|u_0\|_{L^s}+   C\int_{0}^{t} \|u(\tau)\|^{p-\frac{s}{q}}_{L^{\infty}} \|u(\tau)\|^{\frac{s}{q}}_{L^{s}} \, \mathrm{d}\tau\notag\\
		&\leq \|u_0\|_{L^s} + C \|u\|^{p}_{E_T}  T\notag\\
		&\leq R+ C 2^pR^p  T\notag\\
		&\leq 2R,
				\end{align}
	for all $t\in (0,T)$, for a sufficiently small $T>0$ (depending on $R$). Moreover, by \eqref{semigroup1}-\eqref{semigroup2}, we have
	\begin{align*}
		\|\Lambda(u)(t)\|_{L^{\infty}} & \leq \|S_{\beta}(t)u_0\|_{L^{\infty}} + \left\| \int_{0}^{t}S_{\beta}(t-\tau)  I_\alpha(|u|^p)(\tau) \,\mathrm{d}\tau \right\|_{L^{\infty}}  \\
				&\leq \|u_0\|_{L^\infty}+C   \int_{0}^{t} (t-\tau)^{-\frac{n}{\beta r}} \|I_\alpha (|u|^p)(\tau)\|_{L^{r}} \, \mathrm{d}\tau,
				\end{align*}	
		for all $t\in (0,T)$.	Using Lemma \ref{Hardy} with $1/\tilde{q}=\alpha/n+1/r$, we get
		\begin{align*}
		\|\Lambda(u)(t)\|_{L^{\infty}} &\leq \|u_0\|_{L^\infty}+   C\int_{0}^{t} (t-\tau)^{-\frac{n}{\beta r}}\||u(\tau)|^p\|_{L^{\tilde{q}}} \, \mathrm{d}\tau.
				\end{align*}	
			By choosing $1/\tilde{q}<\min\{p/s,\,(\beta+\alpha)/n\}$, we infer that $n/(\beta r)<1$ and $p>s/\tilde{q}$. So,
		\begin{align}\label{fix2}
		\|\Lambda(u)(t)\|_{L^{\infty}} &\leq \|u_0\|_{L^\infty}+  C \int_{0}^{t}(t-\tau)^{-\frac{n}{\beta r}} \|u(\tau)\|^{p-\frac{s}{\tilde{q}}}_{L^{\infty}} \||u(\tau)|^{\frac{s}{\tilde{q}}}\|_{L^{\tilde{q}}} \, \mathrm{d}\tau\notag\\
		&= \|u_0\|_{L^\infty}+   C\int_{0}^{t}(t-\tau)^{-\frac{n}{\beta r}} \|u(\tau)\|^{p-\frac{s}{\tilde{q}}}_{L^{\infty}} \|u(\tau)\|^{\frac{s}{\tilde{q}}}_{L^{s}} \, \mathrm{d}\tau\notag\\
		&\leq \|u_0\|_{L^\infty} + C \|u\|^{p}_{E_T}  T^{1-\frac{n}{\beta r}}\notag\\
		&\leq R+ C 2^pR^p  T^{1-\frac{n}{\beta r}}\notag\\
		&\leq 2R,
				\end{align}
	for all $t\in (0,T)$, for a sufficiently small $T>0$ (depending on $R$). By combining \eqref{fix1}-\eqref{fix2}, we conclude that $\|\Lambda(u)\|_{E_T}\leq 2 R$, that is  $\Lambda(u) \in B_{T}(R)$.\\
	Similarly, one can see that $\Lambda$ is a contraction. For $u,v \in B_T(R)$, we have
$$
		\|\Lambda(u)(t)-\Lambda(v)(t)\|_{L^{s}}  \leq \int_{0}^{t} \|I_\alpha (||u|^{p}-|v|^{p}|)(\tau)\|_{L^s} \, \mathrm{d}\tau,\qquad\hbox{for all}\,\, t\in (0,T). 
	$$
Using Lemma \ref{Hardy} with $1/q=\alpha/n+1/s$, we get
		\begin{align*}
		\|\Lambda(u)(t)-\Lambda(v)(t)\|_{L^{s}} &\leq   C\int_{0}^{t} \||u(\tau)|^{p}-|v(\tau)|^{p}\|_{L^{q}} \, \mathrm{d}\tau.
				\end{align*}	
		As 
		$$||u(\tau)|^{p}-|v(\tau)|^{p}|\leq C(|u(\tau)|^{p-1}+|v(\tau)|^{p-1})|u(\tau)-v(\tau)|$$
		 and 
		 $$\frac{1}{q}=\frac{\alpha}{n}+\frac{1}{s},$$
			by applying H\"older's inequality, we obtain	
			$$\||u(\tau)|^{p}-|v(\tau)|^{p}\|_{L^{q}}\leq C(\||u(\tau)|^{p-1}\|_{L^{\frac{n}{\alpha}}}+\||v(\tau)|^{p-1}\|_{L^{\frac{n}{\alpha}}})\|u(\tau)-v(\tau)\|_{L^{s}}.$$	
				Since $s<n(p-1)/\alpha$ implies $p-1>s\alpha/n$, we conclude that
				$$\||u(\tau)|^{p-1}\|_{L^{\frac{n}{\alpha}}}\leq \|u(\tau)\|^{p-1-\frac{s\alpha}{n}}_{L^{\infty}}\|u(\tau)\|^{\frac{s\alpha}{n}}_{L^s}\leq (2R)^{p-1},$$
			and
				$$\||v(\tau)|^{p-1}\|_{L^{\frac{n}{\alpha}}}\leq \|v(\tau)\|^{p-1-\frac{s\alpha}{n}}_{L^{\infty}}\|v(\tau)\|^{\frac{s\alpha}{n}}_{L^s}\leq (2R)^{p-1}.$$
			Therefore,
$$
			\||u(\tau)|^{p}-|v(\tau)|^{p}\|_{L^{q}}\leq  C 2^p\,R^{p-1}\|u(\tau)-v(\tau)\|_{L^{s}}.
		$$
			It follows that
						\begin{equation}\label{fix3}
		\|\Lambda(u)(t)-\Lambda(v)(t)\|_{L^{s}} \leq C 2^pR^{p-1} T\|u-v\|_{E_T}\leq \frac{1}{2}d(u,v),
				\end{equation}
	for all $t\in (0,T)$, for a sufficiently small $T>0$ (depending on $R$). Moreover, 
	\begin{align*}
		\|\Lambda(u)(t)-\Lambda(v)(t)\|_{L^{\infty}} & =\left\| \int_{0}^{t}S_{\beta}(t-\tau)  I_\alpha(|u|^{p}-|v|^{p})(\tau) \,\mathrm{d}\tau \right\|_{L^{\infty}}  \\
				&\leq C   \int_{0}^{t} (t-\tau)^{-\frac{n}{\beta r}} \|I_\alpha (||u|^{p}-|v|^{p}|)(\tau)\|_{L^{r}} \, \mathrm{d}\tau,
				\end{align*}	
		for all $t\in (0,T)$.	Using Lemma \ref{Hardy} with $1/\tilde{q}=\alpha/n+1/r$, we get
		\begin{align*}
		\|\Lambda(u)(t)-\Lambda(v)(t)\|_{L^{\infty}} &\leq  C\int_{0}^{t} (t-\tau)^{-\frac{n}{\beta r}}\||u(\tau)|^{p}-|v(\tau)|^{p}\|_{L^{\tilde{q}}} \, \mathrm{d}\tau.
				\end{align*}	
			As 
		$$|u(\tau)|^{p}-|v(\tau)|^{p}\leq C(|u(\tau)|^{p-1}+|v(\tau)|^{p-1})|u(\tau)-v(\tau)|$$
		 and 
		 $$\frac{1}{\tilde{q}}=\frac{\alpha}{n}+\frac{1}{r},$$
			by applying H\"older's inequality, we obtain	
			$$\||u(\tau)|^{p-1}-|v(\tau)|^{p-1}\|_{L^{\tilde{q}}}\leq C(\||u(\tau)|^{p-1}\|_{L^{\frac{n}{\alpha}}}+\||v(\tau)|^{p-1}\|_{L^{\frac{n}{\alpha}}})\|u(\tau)-v(\tau)\|_{L^{r}}.$$	
				Since $s<n(p-1)/\alpha$ implies $p-1>s\alpha/n$, we conclude that
				$$\||u(\tau)|^{p-1}\|_{L^{\frac{n}{\alpha}}}\leq \|u(\tau)\|^{p-1-\frac{s\alpha}{n}}_{L^{\infty}}\|u(\tau)\|^{\frac{s\alpha}{n}}_{L^s}\leq (2R)^{p-1},$$
			and
				$$\||v(\tau)|^{p-1}\|_{L^{\frac{n}{\alpha}}}\leq \|v(\tau)\|^{p-1-\frac{s\alpha}{n}}_{L^{\infty}}\|v(\tau)\|^{\frac{s\alpha}{n}}_{L^s}\leq (2R)^{p-1}.$$
				Furthermore, by choosing $1/\tilde{q}<\min\{p\alpha/n(p-1),\, (\beta+\alpha)/n\}$, we infer that $n/(\beta r)<1$ and $r>n(p-1)/\alpha>s$, so
				$$\|u(\tau)-v(\tau)\|_{L^{r}}\leq \|u(\tau)-v(\tau)\|^{1-\frac{s}{r}}_{L^\infty}\|u(\tau)-v(\tau)\|^{\frac{s}{r}}_{L^s}\leq \|u-v\|_{E_T}.$$
			Noting that $1/\tilde{q}<\min\{p\alpha/n(p-1),\, (\beta+\alpha)/n\}<\min\{p/s,\,(\beta+\alpha)/n\}$, due to $s<n(p-1)/\alpha$. Therefore,
$$
			\||u(\tau)|^{p}-|v(\tau)|^{p}\|_{L^{\tilde{q}}}\leq C2^p\,R^{p-1}\|u-v\|_{E_T}.
		$$
		It follows that
		\begin{equation}\label{fix4}
		\|\Lambda(u)(t)-\Lambda(v)(t)\|_{L^{\infty}} \leq C 2^pR^{p-1} T^{1-\frac{n}{\beta r}}\|u-v\|_{E_T}\leq \frac{1}{2}d(u,v),
				\end{equation}
	for all $t\in (0,T)$, for a sufficiently small $T>0$ (depending on $R$). By combining \eqref{fix3}-\eqref{fix4}, we conclude that 
$$
		d(\Lambda(u),\Lambda(v)) \leq \frac{1}{2}d(u,v).
$$
	Consequently, by the contraction principle, $\Lambda$  has a unique fixed point $u$ in $B_T(R)$.\\
	
	\noindent {\it Step 2. Uniqueness in $E_T$.} We are going  to extend this uniqueness to $E_T$. Let $u,v \in E_T$.  Then, for sufficiently large $\tilde{R}>0$ and small $\tilde{T}>0$, we have $u,v \in B_{\tilde{T}}(\tilde{R})$. As a result, $u(t)= v(t)$ for small $t>0$, which, by a standard continuation argument, extends to the entire interval $(0,T)$. See, e.g., \cite{FinoViana} for further details.\\
	
	\noindent {\it Step 3. Regularity.} One can easily check that $f\in L^1((0,T),L^s(\mathbb{R}^{n}))$, with
	$$f(t):=I_\alpha(|u|^{p})(t),\qquad \hbox{for all}\,\,t\in (0,T).$$
	Applying \cite[Lemma~4.1.5]{CH}, using the continuity of the semigroup $S_{\beta}(t)$,  we conclude that $u\in C([0,T],L^s(\mathbb{R}^{n}))$.\\
	
	\noindent {\it Step 4. Maximal solution and blow-up alternative.} Using the uniqueness of the mild solution, we conclude the existence
	of a solution on a maximal interval $[0,T_{\max})$ where
	\[
	T_{\max}:=\sup\left\{T>0\;;\;\text{$\exists$ mild solution $u\in C([0,T],L^s(\mathbb{R}^{n}))\cap L^\infty((0,T),L^\infty(\mathbb{R}^{n}))$
		to \eqref{44}}\right\}.
	\]
	Obviously, $T_{\max}\leq \infty$. To prove that $\|u(t)\|_{L^s(\mathbb{R}^{n})\cap L^\infty(\mathbb{R}^{n})}\rightarrow\infty$ as $t\rightarrow T_{\max},$
	whenever $T_{\max}<\infty,$ we proceed by contradiction. If
	$$\liminf_{t\rightarrow T_{\max}}\|u(t)\|_{L^s(\mathbb{R}^{n})\cap L^\infty(\mathbb{R}^{n})}=:L<\infty,$$
	then there exists a time sequence $\{t_m\}_{m\geq0}$ tending to $T_{\max}$ as $m\rightarrow\infty$ and such that
	$$\sup_{m\in\mathbb{N}}\|u(t_m)\|_{L^s\cap L^\infty}\leq L+1.$$
	Using again a fixed-point argument with $u(t_m)$ as initial condition, one can deduce that there exists $T(L + 1) > 0$, depends on $L+1$, such that the solution
	$u(t)$ can be extended on the interval $[t_m, t_m + T(L + 1)]$ for any $m\geq0$. Thus, by the definition of the maximality time, $T_{\max}\geq t_m+T(L+1)$, for any $m\geq0$. We get the desired contradiction by letting $m\rightarrow\infty$.\\
	
	\noindent {\it Step 5. Positivity of solutions.} Assume $u_0\geq0$.  In this case, we can construct a nonnegative solution on some interval $[0,T]$ by applying the fixed point argument within the positive cone $E_T^+=\{u\in E_T;\;u\geq0\}$, and using the nonnegativity of $S_\beta(t)$ (see \eqref{P_1+}). By the uniqueness of the solutions, it then follows that $u(t)\geq0$ for all $t\in(0,T_{\max}).$\\

\end{proof}


\section{Blow-up}\label{bll}

\begin{proof}[Proof of Theorem \ref{global}-(i)]
The proof is by contradiction. Assume that $u$ is a global mild solution of \eqref{44}. Then $u$ satisfies the following weak formulation
\begin{eqnarray*}
&{}&\int_0^T\int_{\mathbb{R}^n}I_\alpha(u^p)\psi(t,x)\,dx\,dt+\int_{\mathbb{R}^n}u(0,x)\psi(0,x)\,dx\\
&{}&=\int_0^T\int_{\mathbb{R}^n}u\,(-\Delta)^{\frac{\beta}{2}}\psi(t,x)\,dx\,dt-\int_0^T\int_{\mathbb{R}^n}u\,\psi_t(t,x)\,dx\,dt
\end{eqnarray*}
 for all $T>0$ and all compactly supported $\psi\in C^{2,1}([0,T]\times\mathbb{R}^n)$ such that $\psi(T,\cdotp)=0$.
 Let $R$ and $T$ be large parameters in $(0,\infty)$. Let us choose 
$$\psi(t,x):= \varphi^\ell_R(x) \varphi^\ell_T(t),$$
where
$$
\varphi_R(t)= \Phi\left(\frac{|x|}{R}\right),\qquad \varphi_T(t)= \Phi\left(\frac{t}{T}\right),\qquad x\in\mathbb{R}^n,\,\, t>0,
$$
with $\ell=(2p-1)/(p-1)$, and $\Phi\in \mathcal{C}^\infty(\mathbb{R})$ is a smooth non-increasing function 
satisfying $\mathbbm{1}_{(-\infty,\frac{1}{2}]} \leq \Phi\leq \mathbbm{1}_{(-\infty,1]}$. Then,
\begin{eqnarray}\label{3}
\int_0^T\int_\mathcal{B}I_\alpha(|u|^p)\psi(t,x)\,dx\,dt+\int_{\mathcal{B}}u_0(x)\varphi^{\ell}_R(x)\,dx&=&\int_0^T\int_{\mathcal{B}}u\,\varphi^\ell_T(t)(-\Delta)^{\frac{\beta}{2}}\varphi^{\ell}_R(x)\,dx\,dt\nonumber\\
&{}&\quad -\int_{\frac T2}^T\int_{\mathcal{B}}u\,\varphi^{\ell}_R(x)\partial_t(\varphi^{\ell}_T(t))\,dx\,dt\nonumber\\
&=:&I_1+I_2,
\end{eqnarray}
where $\mathcal{B}=\left\{x\in\mathbb{R}^n;\,\,|x|\leq R\right\}$. Let us first derive an estimate for $I_1$. Using  H\"older's inequality together with  Ju's inequality $(-\Delta)^{\beta/2}\left(\varphi_{R}^\ell\right)\leq \ell\varphi_{R}^{\ell-1}(-\Delta)^{\beta/2}\varphi_{R}$ (see e.g. \cite[Appendix]{Fino5}), we have
\begin{eqnarray}\label{4}
I_1&=&\int_0^T\int_{\mathcal{B}}u\,\varphi^\ell_T(t)(-\Delta)^{\frac{\beta}{2}}\varphi^{\ell}_R(x)\,dx\,dt\nonumber\\
&\leq&\ell \int_0^T\int_{\mathcal{B}}u\,\psi^{\frac{1}{p}}(t,x)\psi^{-\frac{1}{p}}(t,x)\varphi^\ell_T(t)\varphi^{\ell-1}_R(x)\left|(-\Delta)^{\frac{\beta}{2}}\varphi_R(x)\right|\,dx\,dt\nonumber\\
&\leq&\ell\left(\int_0^T\int_{\mathcal{B}}u^{p}\psi(t,x)\,dx\,dt\right)^{\frac{1}{p}}\left(\int_0^T\int_{\mathcal{B}}\varphi^{\ell}_T(t)\varphi_R(x)\,\left|(-\Delta)^{\frac{\beta}{2}}\varphi_R(x)\right|^{\frac{p}{p-1}}\,dx\,dt\right)^{\frac{p-1}{p}} .
\end{eqnarray}
Similarly, by $\partial_t\varphi^\ell_T(t)=\ell\varphi_T^{\ell-1}(t)\partial_t\varphi_T(t)$, we obtain
\begin{equation}\label{5}
I_2\leq\ell \left(\int_{\frac T2}^T\int_{\mathcal{B}}u^{p}\psi(t,x)\,dx\,dt\right)^{\frac{1}{p}}\left(\int_{\frac T2}^T\int_{\mathcal{B}}\varphi_T(t)\varphi_R^{\ell}(x)\,\left|\partial_t\varphi_T(t)\right|^{\frac{p}{p-1}}\,dx\,dt\right)^{\frac{p-1}{p}}.
\end{equation}
Inserting \eqref{4}-\eqref{5} into \eqref{3}, we arrive at
\begin{eqnarray}\label{6}
&{}&\int_0^T\int_{\mathcal{B}}I_\alpha(u^p)\psi(t,x)\,dx\,dt+\int_{\mathcal{B}}u_0(x)\varphi^{\ell}_R(x)\,dx\nonumber\\
&{}&\leq J_1\,\left(\int_0^T\int_{\mathcal{B}}u^{p}\psi(t,x)\,dx\,dt\right)^{\frac{1}{p}}+\,J_2\,\left(\int_{\frac T2}^T\int_{\mathcal{B}}u^{p}\psi(t,x)\,dx\,dt\right)^{\frac{1}{p}}\qquad
\end{eqnarray}
where
$$J_1:=\ell\left(\int_0^T\int_{\mathcal{B}}\varphi^{\ell}_T(t)\varphi _R(x)\,\left|(-\Delta)^{\frac{\beta}{2}}\varphi_R(x)\right|^{\frac{p}{p-1}}\,dx\,dt\right)^{\frac{p-1}{p}},$$
and
$$J_2:=\ell\left(\int_{\frac T2}^T\int_{\mathcal{B}}\varphi _T(t)\varphi^{\ell}_R(x)\,\left|\partial_t\varphi_T(t)\right|^{\frac{p}{p-1}}\,dx\,dt\right)^{\frac{p-1}{p}} .$$
Let us estimate $J_2$. We have
\begin{eqnarray}\label{7}
J_2&\leq& \ell\left(\int_{\mathcal{B}}\varphi^{\ell}_R(x)\,dx\right)^{\frac{p-1}{p}}\left(\int_0^T\Phi\left(\frac{t}{T}\right)\,\left|\partial_t\Phi\left(\frac{t}{T}\right)\right|^{\frac{p}{p-1}}\,dt\right)^{\frac{p-1}{p}}\notag\\
&\leq&C\,R^{\frac{n(p-1)}{p}}T^{-\frac{1}{p}}\left(\int_0^1\Phi(\tilde{t})\,\left|\Phi^{\prime}(\tilde{t})\right|^{\frac{p}{p-1}}\,d\widetilde{t}\right)^{\frac{p-1}{p}}\notag\\
&\leq& C\,R^{\frac{n(p-1)}{p}}T^{-\frac{1}{p}}.
\end{eqnarray}
where we have used the change of variables
$$\widetilde{x}=\frac{x}{R},\qquad \widetilde{t}=\frac{t}{T}.$$
Similarly, by using Lemma \ref{lemma4}, we have
\begin{equation}\label{8}\begin{split}
J_1&= \ell\left(\int_0^T\varphi^\ell_T(t)\,dt\right)^{\frac{p-1}{p}}\left(\int_{\mathcal{B}}\varphi_R(x)\,\left|(-\Delta)^{\frac \beta 2}\varphi_R(x)\right|^{\frac{p}{p-1}}\,dx\right)^{\frac{p-1}{p}}\\&\leq C\,T^{\frac{p-1}{p}}R^{\frac{n(p-1)}{p}-\beta}.\end{split}
\end{equation}
By \eqref{7}-\eqref{8}, we get from \eqref{6} that
\begin{equation}\label{9}\begin{split}
\int_0^T\int_{\mathcal{B}}I_\alpha(u^p)\psi(t,x)\,dx\,dt+\int_{\mathcal{B}}u_0(x)\varphi^{\ell}_R(x)\,dx& \leq C\,T^{\frac{p-1}{p}}R^{\frac{n(p-1)}{p}-\beta}\textbf{I}^{\frac{1}{p}}\\&+C\, R^{\frac{n(p-1)}{p}}T^{-\frac{1}{p}}\textbf{J}^{\frac{1}{p}},\end{split}
\end{equation}
where
\[
\textbf{I}:= \int_0^T\int_{\mathcal{B}}u^p\psi(t,x)\,dx\,dt\qquad\hbox{and}\qquad\textbf{J}:= \int_{\frac T2}^T\int_{\mathcal{B}}u^p\psi(t,x)\,dx\,dt.
\]
To estimate the first term in the left-hand side of \eqref{9}, we have
$$I_\alpha(u^p)=|x|^{-(n-\alpha)}\ast u^p=\int_{\mathbb{R}^n}\frac{(u(t,\xi))^p}{ |x-\xi|^{n-\alpha}}\,d\xi\geq \int_{\mathcal{B}}\frac{(u(t,\xi))^p}{ |x-\xi|^{n-\alpha}}\,d\xi.$$
Note that $|x|\leq R/2$ on $\overline{\mathcal{B}}$ and $|\xi| \leq R$ on $\mathcal{B}$, where $\overline{\mathcal{B}}=\left\{x\in\mathbb{R}^n;\,\,|x|\leq R/2\right\}$,  then
\begin{equation}\label{triangle inequality}
|x-\xi|^{n-\alpha}\leq 2^{n-\alpha} R^{n-\alpha},\quad\hbox{for all}\,\,x\in\overline{\mathcal{B}},\,\xi\in\mathcal{B}.\end{equation}
So,
$$I_\alpha(u^p)(t,x)\geq  \frac{R^{-(n-\alpha)}}{2^{n-\alpha}}\int_{\mathcal{B}}(u(t,\xi))^p\,d\xi,\quad\hbox{for all}\,\,t\in(0,T),\,\,x\in\overline{\mathcal{B}}.$$
Then
\begin{eqnarray}\label{11}
\int_0^T\int_{\mathcal{B}}I_\alpha(u^p) \psi(t,x)\,dx\,dt&\geq&  \frac{R^{-(n-\alpha)}}{2^{n-\alpha}}\int_0^T\int_{\overline{\mathcal{B}}}\int_{\mathcal{B}}(u(t,\xi))^p\psi(t,x)\,d\xi\,dx\,dt\nonumber\\
&\geq&\hbox{meas$(\overline{\mathcal{B}})$}\frac{R^{-(n-\alpha)}}{2^{n-\alpha}}\int_0^T\int_{\mathcal{B}}(u(t,\xi))^p\psi(t,\xi)\,d\xi\,dt\nonumber\\
&=&CR^{\alpha}\,\textbf{I},
\end{eqnarray}
where we have used the fact that $\varphi_R\equiv 1$ on $\overline{\mathcal{B}}$, and $1\geq \varphi_R$ on $\mathcal{B}$.  Combining \eqref{9} and \eqref{11}, we infer that
\begin{equation}\label{10}
CR^{\alpha}\,\textbf{I}+\int_{\mathcal{B}}u_0(x)\varphi^{\ell}_R(x)\,dx \leq C\,T^{\frac{p-1}{p}}R^{\frac{n(p-1)}{p}-\beta}\textbf{I}^{\frac{1}{p}}+C R^{\frac{n(p-1)}{p}}T^{-\frac{1}{p}}\textbf{J}^{\frac{1}{p}}.
\end{equation}
We consider two separate cases. For $p<p_{\mathrm{Fuj}}(n,\beta,\alpha)$, choosing $R=T^{ 1/\beta}$, it follows from \eqref{10} and using $\textbf{J}\leq \textbf{I}$ and $u_0\geq0$, that
$$
T^{\frac{\alpha}{\beta}}\textbf{I}\leq C\,T^{\frac{n(p-1)-\beta}{\beta p}}\textbf{I}^{\frac{1}{p}},
$$
that is,
\begin{equation}\label{14}
\textbf{I}\leq C\,T^{\frac{(n-\alpha)p -n-\beta}{\beta (p-1)}}.
\end{equation}
Consequently, by letting $T\rightarrow\infty$ in \eqref{14}, and using the fact that $$p<p_{\mathrm{Fuj}}(n,\beta,\alpha)\Leftrightarrow (n-\alpha)p -n-\beta<0,$$ and the monotone convergence theorem, we arrive at 
$$\int_0^\infty\int_{\mathbb{R}^n}u^p\,dx\,dt\leq 0,$$
that is $u=0$ a.e., which, using \eqref{10}, implies that
$$0<\int_{\mathbb{R}^n}u_0(x)\,dx\leq 0;$$
contradiction.  In order to get a contradiction in the critical case $p=p_{\mathrm{Fuj}}(n,\beta,\alpha)$ too, we apply the same change of variables as before, using $u_0\geq 0$, we obtain from \eqref{10},
$$
\textbf{I}\leq C\,T^{\frac{p-1}{p}}R^{\frac{n(p-1)}{p}-\beta-\alpha}\textbf{I}^{\frac{1}{p}}+\,C R^{\frac{n(p-1)}{p}-\alpha}T^{-\frac{1}{p}}\textbf{J}^{\frac{1}{p}}.
$$
By setting $R=(KT)^{\frac 1\beta}$, where $K\ge 1$ and $K<T$, so that $T$ and $K$ cannot simultaneously tend to infinity, and considering the fact that $p=p_{\mathrm{Fuj}}(n,\beta,\alpha)$, we conclude that
$$
\textbf{I}\leq C\, K^{-\frac{\beta+\alpha}{\beta+n}}\textbf{I}^{\frac{1}{p}}+\, CK^{\frac{n-\alpha}{\beta+n}}\textbf{J}^{\frac{1}{p}}.
$$
Therefore, by using $\varepsilon$-Young's inequality $C\,K^{-\frac{\beta+\alpha}{\beta+n}}\textbf{I}^{\frac{1}{p}}\leq \varepsilon \textbf{I}+C_\varepsilon K^{-1}$ with $\varepsilon<1$, it follows that
\begin{equation}\label{weak2C}
(1-\varepsilon)\textbf{I}\leq C\, K^{-1}+C\,K^{\frac{n-\alpha}{\beta+n}}\textbf{J}^{\frac{1}{p}}.
	\end{equation}
	On the other hand, from \eqref{14} as $T\rightarrow\infty$, and taking into account that $p=p_{\mathrm{Fuj}}(n,\beta,\alpha)$, it follows that
	\begin{equation}\label{regularityA}
		u\in L^p((0,\infty),L^p(\mathbb{R}^{n})).
	\end{equation}
	Taking the limit as $T\rightarrow\infty$ in \eqref{weak2C}, and applying \eqref{regularityA} along with the Lebesgue dominated convergence theorem, we conclude that
	$$\int_0^\infty\int_{\mathbb{R}^n}u^p(t,x)\,dx\,dt\lesssim K^{-1}.$$
	Therefore, taking a sufficiently large $K$ we obtain the desired contradiction.\end{proof}
\section{Nonexistence}
\begin{proof}[Proof of Theorem \ref{theo11}] (i) Repeating a similar computation as in Section \ref{bll}, we obtain
\begin{equation}\label{900}\begin{split}
\int_0^T\int_{\mathcal{B}}(\mathcal{K}\ast|u|^p)\psi(t,x)\,dx\,dt+\int_{\mathcal{B}}u_0(x)\varphi^{\ell}_R(x)\,dx& \leq C\,T^{\frac{p-1}{p}}R^{\frac{n(p-1)}{p}-\beta}\textbf{I}^{\frac{1}{p}}\\&+C\, R^{\frac{n(p-1)}{p}}T^{-\frac{1}{p}}\textbf{J}^{\frac{1}{p}}.\end{split}
\end{equation}
 To estimate the first term in the left-hand side of \eqref{900}, we have
$$(\mathcal{K}\ast|u|^p)(x)=\int_{\mathbb{R}^n}\mathcal{K}(|x-y|)|u(y)|^p\,dy\geq \int_{\mathcal{B}}\mathcal{K}(|x-y|)|u(y)|^p\,dy.$$
Note that $|x|\leq R/2$ on $\overline{\mathcal{B}}$ and $|\xi| \leq R$ on $\mathcal{B}$, where $\overline{\mathcal{B}}=\left\{x\in\mathbb{R}^n;\,\,|x|\leq R/2\right\}$,  then
$$|x-\xi|^{n-\alpha}\leq 2^{n-\alpha} R^{n-\alpha},\quad\hbox{for all}\,\,x\in\overline{\mathcal{B}},\,\xi\in\mathcal{B},$$
which implies that
$$\mathcal{K}(|x-y|)\geq \mathcal{K}(2R),\quad\hbox{for all}\,\,x\in\overline{\mathcal{B}},\,\xi\in\mathcal{B},$$
and for all $R\gg1$, namely $2R>R_0$. So,
$$(\mathcal{K}\ast|u|^p)(t,x)\geq  \mathcal{K}(2R)\int_{\mathcal{B}}|u(y)|^p\,dy,\quad\hbox{for all}\,\,t\in(0,T),\,\,x\in\overline{\mathcal{B}}.$$
Then
\begin{eqnarray}\label{111}
\int_0^T\int_{\mathcal{B}}(\mathcal{K}\ast|u|^p) \psi(t,x)\,dx\,dt&\geq& \mathcal{K}(2R)\int_0^T\int_{\overline{\mathcal{B}}}\int_{\mathcal{B}}(u(t,\xi))^p\psi(t,x)\,d\xi\,dx\,dt\nonumber\\
&\geq&\hbox{meas$(\overline{\mathcal{B}})$}\mathcal{K}(2R)\int_0^T\int_{\mathcal{B}}(u(t,\xi))^p\psi(t,\xi)\,d\xi\,dt\nonumber\\
&=&CR^{n}\mathcal{K}(2R)\,\textbf{I},
\end{eqnarray}
where we have used the fact that $\varphi_R\equiv 1$ on $\overline{\mathcal{B}}$, and $1\geq \varphi_R$ on $\mathcal{B}$.  Combining \eqref{900} and \eqref{111}, we infer that
\begin{equation}\label{101}
CR^{n}\mathcal{K}(2R)\,\textbf{I}+\int_{\mathcal{B}}u_0(x)\varphi^{\ell}_R(x)\,dx \leq C\,T^{\frac{p-1}{p}}R^{\frac{n(p-1)}{p}-\beta}\textbf{I}^{\frac{1}{p}}+C R^{\frac{n(p-1)}{p}}T^{-\frac{1}{p}}\textbf{J}^{\frac{1}{p}}.
\end{equation}
By choosing $R=T^{\frac 1\beta}$, it follows from \eqref{101} and using $\textbf{J}\leq \textbf{I}$ and $u_0\geq0$, that
$$
T^{\frac{n}{\beta}}\mathcal{K}(2T^{\frac{1}{\beta}})\textbf{I}\leq C\,T^{\frac{n(p-1)-\beta}{\beta p}}\textbf{I}^{\frac{1}{p}},
$$
that is,
\begin{equation}\label{141}
\textbf{I}^{\frac{p-1}{p}}\leq \frac{C}{\mathcal{K}(2\,T^{\frac 1\beta})\,T^{\frac{n+\beta}{\beta p}}}.
\end{equation}
As
$$\limsup_{R\rightarrow\infty}\left(\mathcal{K}(R)\,R^{\frac{n+\beta}{p}}\right)>0,$$
there exists a sequence $\{R_j\}_j$ such that 
\begin{equation}\label{161}
R_j\rightarrow+\infty\qquad\hbox{and}\qquad \mathcal{K}(R_j)\,R_j^{\frac{n+\beta}{p}}\longrightarrow\ell>0,\quad\hbox{as}\,\,j\rightarrow\infty.
\end{equation}
Without loss of generality, we may assume that $R_j>R_{j-1}$ for all $j>1$.\\
\noindent {\bf If $\ell=\infty$}, replacing $T$ by $(R_j/2)^\beta$, we have
$$
\textbf{I}^{\frac{p-1}{p}}\leq \frac{C}{\mathcal{K}(R_j)\,R_j^{\frac{n+\beta}{ p}}}.
$$
Consequently, passing to the limit when $j\rightarrow\infty$, using \eqref{161}, and the monotone convergence theorem,  we arrive at 
$$\int_0^\infty\int_{\mathbb{R}^n}u^p\,dx\,dt\leq 0,$$
that is $u=0$ a.e., which, using \eqref{101}, $u_0\in L^1(\mathbb{R}^n)$, and the dominated convergence theorem, implies that
$$0<\int_{\mathbb{R}^n}u_0(x)\,dx\leq 0;$$
contradiction.\\
\noindent {\bf If $\ell<\infty$}, by setting
 $$R=\frac{R_j}{2}\qquad \hbox{and}\qquad T=\left(\frac{R_j}{2}\right)^\beta K^{-1},$$
  where $K\ge 1$ and $K<R_j$, so that $R_j$ and $K$ cannot simultaneously tend to infinity, using $u_0\geq0$, we conclude from \eqref{101} that 
$$\textbf{I}\leq C\frac{K^{\frac{1-p}{p}}}{\mathcal{K}(R_j)R_j^{\frac{n+\beta}{p}}}\textbf{I}^{\frac{1}{p}}+C\frac{K^{\frac{1}{p}}}{\mathcal{K}(R_j)R_j^{\frac{n+\beta}{p}}}\textbf{J}^{\frac{1}{p}}.$$
By Young's inequality, we conclude that
$$\textbf{I}\leq C\frac{K^{-1}}{\left(\mathcal{K}(R_j)R_j^{\frac{n+\beta}{p}}\right)^{\frac{p}{p-1}}}+C\frac{K^{\frac{1}{p}}}{\mathcal{K}(R_j)R_j^{\frac{n+\beta}{p}}}\textbf{J}^{\frac{1}{p}}.$$
Since $\ell<\infty$, then \eqref{141} shows that $u\in L^p((0,\infty),L^p(\mathbb{R}^{n}))$. This implies that $\textbf{J} \longrightarrow 0$, when $j\rightarrow\infty$. Again, passing to the limit when $j\rightarrow\infty$, using \eqref{161} and $\ell\in(0,\infty)$, along with the Lebesgue dominated convergence theorem, we conclude that
	$$\int_0^\infty\int_{\mathbb{R}^n}u^p(t,x)\,dx\,dt\lesssim K^{-1}.$$
	Therefore, taking a sufficiently large $K$ we obtain the desired contradiction.\\
   \noindent (ii) Repeating exactly the same steps as in case (i), we obtain
   $$
C\,\mathcal{K}(2R)\textbf{I}+R^{-n}\int_{\mathcal{B}}u_0(x)\varphi^{\ell}_R(x)\,dx \leq CT^{\frac{p-1}{p}}R^{-\frac{n}{p}-\beta}\textbf{I}^{\frac{1}{p}}+CR^{-\frac{n}{p}}T^{-\frac{1}{p}}\textbf{J}^{\frac{1}{p}}.
$$
Since $\mathcal{K}(2R)\geq \mathcal{K}(R/2)$ whenever $2R>R_0$, it follows that
 $$
C\,\mathcal{K}(R/2)\textbf{I}+R^{-n}\int_{\overline{\mathcal{B}}}u_0(x)\,dx \leq CT^{\frac{p-1}{p}}R^{-\frac{n}{p}-\beta}\textbf{I}^{\frac{1}{p}}+CR^{-\frac{n}{p}}T^{-\frac{1}{p}}\textbf{J}^{\frac{1}{p}}.
$$
Therefore,
    \begin{equation}\label{101A}
C\,\textbf{I}+R^{-n}(\mathcal{K}(R/2))^{-1}\int_{\overline{\mathcal{B}}}u_0(x)\,dx \leq C\frac{T^{\frac{p-1}{p}}R^{-\frac{n}{p}-\beta}}{\mathcal{K}(R/2)}\textbf{I}^{\frac{1}{p}}+C\frac{R^{-\frac{n}{p}}T^{-\frac{1}{p}}}{\mathcal{K}(R/2)}\textbf{J}^{\frac{1}{p}}.
\end{equation}
By choosing $R=T^{\frac 1\beta}$, it follows from \eqref{101A}, together with the condition $\textbf{J}\leq \textbf{I}$ and the assumption on the initial data $u_0$, that
\begin{equation}\label{141A}
\textbf{I}^{\frac{p-1}{p}}\leq \frac{C}{\mathcal{K}(T^{\frac 1\beta}/2)\,T^{\frac{n+\beta}{\beta p}}}.
\end{equation}
We conclude the desired result by proceeding in the same manner as in case (i).\\
	\noindent (iii) In this case, as $R>1$, we have the following estimate
\begin{align*}
\int_{\mathcal{B}}u_0(x)\varphi^\ell_{R}(x)\,dx&\geq \int_{\overline{\mathcal{B}}}u_0(x)\,dx\\&\geq\overline{ \varepsilon} \int_{\overline{\mathcal{B}}}(1+|x|^2)^{-\gamma/2}\,dx\\&\geq \overline{\varepsilon}\,C \int_{\overline{\mathcal{B}}}\left(R^{2}+ R^{2}\right)^{-\gamma/2}\,dx\\&=\overline{\varepsilon}\,C R^{n-\gamma}.
\end{align*}
Therefore, by repeating the same calculation as in case (i) (see \eqref{101}), choosing $R=T^{ 1/\beta}$, and using $\textbf{J}\leq \textbf{I}$, we get
$$
T^{\frac{n}{\beta}}\mathcal{K}(2T^{\frac{1}{\beta}})\textbf{I}+\overline{\varepsilon}\,C T^{\frac{n-\gamma}{\beta}}\leq C\,T^{\frac{n(p-1)-\beta}{\beta p}}\textbf{I}^{\frac{1}{p}},
$$
that is,
\begin{align*}
\textbf{I}+\frac{\overline{\varepsilon}\,C}{T^{\frac{\gamma}{\beta}}\mathcal{K}(2T^{\frac{1}{\beta}})}& \leq \frac{C}{T^{\frac{n+\beta}{\beta p}}\mathcal{K}(2T^{\frac{1}{\beta}})}\textbf{I}^{\frac{1}{p}}\\&\leq \frac{C}{\left(T^{\frac{n+\beta}{\beta p}}\mathcal{K}(2T^{\frac{1}{\beta}})\right)^{\frac{p}{p-1}}}+\frac{1}{2}\textbf{I},
\end{align*}
where we have used the following Young's inequality
$$ab\leq \frac{1}{2}a^{p}+Cb^{\frac{p}{p-1}}.$$
This implies that
$$
\overline{\varepsilon}\lesssim \left(\mathcal{K}(2R)^{-1}R^{\gamma(p-1)-n-\beta}\right)^{\frac{1}{p-1}},
$$
and so,
$$
0<\overline{\varepsilon}\lesssim \left(\liminf_{R\rightarrow\infty}\left(\mathcal{K}(R)^{-1}R^{\gamma(p-1)-n-\beta}\right)\right)^{\frac{1}{p-1}}=0;
$$
contradiction. The proof is complete.\end{proof}



\section{Global existence}

\begin{proof}[Proof of Theorem \ref{global}-(ii)] Since $p>p_{\mathrm{Fuj}}(n,\beta,\alpha)$, it is possible to choose a positive constant $q>0$ such that
\begin{equation}\label{estiA}
    \frac{\beta+\alpha}{\beta(p-1)}-\frac{1}{p}<\frac{n}{\beta
    q}< \frac{\beta+\alpha}{\beta(p-1)}\quad \text{with} \quad q> p.
\end{equation}
From this, it follows that 
\begin{equation}\label{estiB}
    q>\frac{n(p-1)}{\beta+\alpha}=q_{\mathrm{sc}}>1.
\end{equation}
We then define
\begin{equation}\label{estiC}\begin{split}
    \beta^*:&=\frac{n}{\beta q_{\mathrm{sc}}}-\frac{n}{\beta
    q}\\&=\frac{\beta+\alpha}{\beta(p-1)}-\frac{n}{\beta
    q}.\end{split}
\end{equation}
Therefore, based on relations (\ref{estiA})-(\ref{estiC}), the following relations hold
\begin{equation}\label{estiD}
    \beta^*>0,\qquad 1-\frac{n(p-1)}{\beta q}+\frac{\alpha}{\beta} -(p-1)\beta^*=0,\qquad\hbox{and}\qquad p\beta^*<1.
\end{equation}
 Since $u_0\in L^{q_{\mathrm{sc}}}(\mathbb{R}^{n})$, applying Lemma \ref{Lp-Lqestimate} with $q>q_{\mathrm{sc}}$, and using \eqref{estiC}, we
get
\begin{equation}\label{estiE}
    \sup_{t>0}t^{\beta^*}\|S_{\beta}(t)u_0\|_{L^q}\leq
   C \|u_0\|_{L^{q_{\mathrm{sc}}}}=:\rho<\infty.
\end{equation}
Set
\begin{equation}\label{estiF}
    \mathbb{X}:=\left\{u\in
    L^\infty((0,\infty),L^q(\mathbb{R}^{n}));\;\sup_{t>0}t^{\beta^*}\|u(t)\|_{L^q}\leq\delta\right\},
\end{equation}
where $\delta>0$ is chosen to be sufficiently small. For $u,v\in \mathbb{X}$, we define the metric
\begin{equation}\label{estiG}
    d_{\mathbb{X}}(u,v):=\sup_{t>0}t^{\beta^*}\|u(t)-v(t)\|_{L^q}.
\end{equation}
It is straightforward to verify that $(\mathbb{X},d)$ is a nonempty complete metric space. For $u\in \mathbb{X}$, we define the mapping $\Phi(u)$ by
\begin{equation}\label{estiH}
    \Phi(u)(t):=S_{\beta}(t)u_0 + \int_{0}^{t}S_{\beta}(t-\tau) I_\alpha(|u|^{p})(\tau) \,\mathrm{d}\tau,\quad \text{for all}\,\, t\geq0.
\end{equation}
Let us now verify that the operator $\Phi:  \mathbb{X} \rightarrow \mathbb{X}$. By employing inequalities \eqref{estiE} and \eqref{estiF}, along with Lemma \ref{Lp-Lqestimate}, we derive the following estimate for any $u \in  \mathbb{X}$, 
$$
    t^{\beta^*}\|\Phi(u)(t)\|_{L^q}\leq \rho+\,Ct^{\beta^*}\int_0^t(t-\tau)^{-\frac{n}{\beta}\left(\frac{1}{r_1}-\frac{1}{q}\right)}\|I_\alpha(|u|^p)(\tau)\|_{L^{r_1}}\,d\tau,
   $$
   for any $1<r_1<q$.  By employing the assumption $q>p$, and using Lemma \ref{Hardy} with $p/q=\alpha/n+1/r_1$, we get
\begin{eqnarray}\label{estiI}
    t^{\beta^*}\|\Phi(u)(t)\|_{L^q}&\leq& \rho+\,Ct^{\beta^*}\int_0^t(t-\tau)^{-\frac{n}{\beta}\left(\frac{p-1}{q}-\frac{\alpha}{n}\right)}\|u(\tau)\|^{p}_{L^{q}}\,d\tau\nonumber\\
    &\leq&\rho+\,C\delta^{p} t^{\beta^*}\int_0^t(t-\tau)^{-\frac{n}{\beta}\left(\frac{p-1}{q}-\frac{\alpha}{n}\right)}\tau^{-\beta^* p}\,d\tau.
\end{eqnarray}
Now, using the parameter constraints in \eqref{estiA} and \eqref{estiD}, together with the condition $p\beta^*<1$, the integral becomes
\begin{equation}\label{estiJ}
\int_0^t(t-\tau)^{-\frac{n}{\beta}\left(\frac{p-1}{q}-\frac{\alpha}{n}\right)}\tau^{-\beta^* p}\,d\tau=C t^{-\beta^*},
\end{equation}
valid for all $t\geq0$. It then follows from estimates \eqref{estiI} and \eqref{estiJ} that
\begin{equation}\label{estiK}
    t^{\beta^*}\|\Phi(u)(t)\|_{L^q}\leq \rho+C\delta^p.
\end{equation}
Hence, if $\rho$ and $\delta$ are chosen sufficiently small such that $\rho+C\delta^p\leq\delta$, it follows that $\Phi(u)\in\mathbb{X}$, that is, $\Phi: \mathbb{X}\rightarrow \mathbb{X}$. A similar argument shows that, under the same smallness assumptions on $\rho$ and $\delta$, the operator $\Phi$ is a strict contraction. Therefore, it has a unique fixed point $u\in \mathbb{X}$ which corresponds to a mild solution of problem \eqref{44}.

We now aim to prove that $u\in L^\infty((0,\infty),L^\infty(\mathbb{R}^{n}))$. We begin by showing that $u\in L^\infty((0,T),L^\infty(\mathbb{R}^{n}))$ for some sufficiently small $T>0$. Indeed, the previous argument ensures uniqueness in the space $ \mathbb{X}_T,$ where for any $T>0,$
$$
 \mathbb{X}_T:=\left\{u\in
    L^\infty((0,T),L^q(\mathbb{R}^{n}));\;\sup_{0<t<T}t^{\beta^*}\|u(t)\|_{L^q}\leq\delta\right\}.
$$
Let $\tilde{u}$ denote the local solution of \eqref{44} established in Theorem \ref{localexistence}. From inequality \eqref{estiB}, we know that $q_{\mathrm{sc}}<q<\infty$, which implies 
$$u_0\in L^\infty(\mathbb{R}^{n})\cap L^{q_{\mathrm{sc}}}(\mathbb{R}^{n})\subset L^\infty(\mathbb{R}^{n})\cap L^{q}(\mathbb{R}^{n}).$$ 
Moreover, using $p>p_{\mathrm{Fuj}}(n,\beta,\alpha)$, we have $$n/(n-\alpha)<q_{\mathrm{sc}}<q<n(p-1)/\alpha.$$ Thus, Theorem \ref{localexistence}  guarantees that $\tilde{u}\in L^\infty((0,T_{\max}),L^\infty(\mathbb{R}^{n})\cap L^q(\mathbb{R}^{n}))$. Consequently, for sufficiently small $T>0$, we have
$$\sup\limits_{t\in(0,T)}t^{\beta^*}\|\tilde{u}(t)\|_{L^q}\leq\delta.$$
Due to the uniqueness of solutions in $\mathbb{X}_T$, it follows that $u=\tilde{u}$ on
$[0,T]$, leading to the conclusion that $u\in L^\infty((0,T),L^\infty(\mathbb{R}^{n})\cap L^q(\mathbb{R}^{n}))$.

To extend the regularity to $[T,\infty)$, we employ a bootstrap argument. For $t>T,$ we express $u(t)$ as
\begin{eqnarray*}
  u(t)-S_{\beta}(t)u_0 &=&
  \int_0^TS_{\beta}(t-\tau)I_\alpha(|u|^{p})(\tau)\,d\tau+\int_T^tS_{\beta}(t-\tau)I_\alpha(|u|^{p})(\tau)\,d\tau\\
   &\equiv& I_1(t)+I_2(t).
\end{eqnarray*}
Since $u\in L^\infty((0,T),L^\infty(\mathbb{R}^{n})\cap L^q(\mathbb{R}^{n}))$, and $n/(n-\alpha)<q<n(p-1)/\alpha$, we immediately have $I_1\in
L^\infty((T,\infty),L^\infty(\mathbb{R}^{n})).$ Furthermore, by similar estimates used in the fixed-point argument and noting that $$t^{-\beta^*}\leq
T^{-\beta^*}<\infty,$$ we also get $I_1\in L^\infty((T,\infty),L^q(\mathbb{R}^{n}))$.
Next, from \eqref{estiB}, we observe that $q>q_{\mathrm{sc}}$, which guarantees the existence of some $r\in(q,\infty]$ such that
\begin{equation}\label{estiL}
\frac{n}{\beta}\left(\frac{p}{q}-\frac{\alpha}{n}-\frac{1}{r}\right)<1.
\end{equation}
For $T<t$, since $u\in
L^\infty((0,\infty),L^q(\mathbb{R}^{n})),$ we obtain
$$u^p\in L^\infty((T,t),L^{\frac{q}{p}}(\mathbb{R}^{n})).$$
 Thus, applying Lemma \ref{Lp-Lqestimate}, condition \eqref{estiL}, and Lemma \ref{Hardy} with $p/q=\alpha/n+1/r_1$, we deduce that $I_2\in
L^\infty((T,\infty),L^r(\mathbb{R}^{n}))$.
Since the terms $S_{\beta}(t)u_0$ and $I_1$ belong to
$$L^\infty((T,\infty),L^\infty(\mathbb{R}^{n}))\cap
L^\infty((T,\infty),L^q(\mathbb{R}^{n}))\subseteq L^\infty((T,\infty),L^r(\mathbb{R}^{n})),$$ we conclude that $u\in
L^\infty((T,\infty),L^r(\mathbb{R}^{n}))$. Iterating this procedure a finite
number of times, eventually leads to $u\in
L^\infty((T,\infty),L^\infty(\mathbb{R}^{n}))$. This completes the proof. 
\end{proof}





\section*{Acknowledgments}
The authors are deeply grateful to Professor Enzo Mitidieri (Italy) and Professor Manuel Del Pino (UK) for their valuable comments and advice, which significantly improved the content of the Introduction. They also wish to thank Professor Philippe Souplet (France) for his insightful recommendations on references concerning parabolic equations with time-nonlocal nonlinearities.

\section*{Funding} Ahmad Fino is supported by the Research Group Unit, College of Engineering and Technology, American University of the Middle East. Berikbol T. Torebek is supported by the Science Committee of the Ministry of Education and Science of the Republic of Kazakhstan (Grant No. AP26195417)

\section*{Declaration of competing interest} The authors declare that there is no conflict of interest.

\section*{Data Availability Statements} The manuscript has no associated data.




\end{document}